\documentclass[12pt] {amsart}
\usepackage{amsthm, mathtools, amsmath,amscd,graphicx,amssymb,
amstext, epsfig,verbatim, multicol,mathrsfs,mathbbol}
\newtheorem{thm}{Theorem}
                                             
                                             \newtheorem*{TheoremA}{Main Theorem}
\newtheorem{lem}{Lemma}
\newtheorem{prop}{Proposition}
{\theoremstyle{remark}
\newtheorem {Rem}{Remark}}
{\theoremstyle{definition}
\newtheorem{example}{Example}}
                         
\def\div{\mathrm{div}}

\def\fchar{\mathrm{char}}

                                                   \newcommand{\CC}{\mathcal C}

\newcommand{\vf}{\varphi}

\newcommand{\eps}{\varepsilon}

\newcommand{\di}{\textrm{div}}
\newcommand{\Di}{\textrm{Div}}
                                             \newcommand{\alb}{\textit{alb}}
\newcommand{\beq}{\begin{equation}}
\newcommand{\eeq}{\end{equation}}

\newcommand{\BQ}{\mathbb{Q}}

\newcommand{\Eta}{\mathrm{H}}
                                     \newcommand{\bea}{\mathbf a}
                                     \newcommand{\beb}{\mathbf b}
                                      \newcommand{\bec}{\mathbf c}
                                       \newcommand{\bed}{\mathbf d}

\title[Torsion Points of small order]{Torsion points of small order on cyclic covers of $\mathbb P^1$}
\author {Boris M. Bekker}
\address{ St.Petersburg State University, 7/9 Universitetskaya nab., St. Petersburg, 199034 Russia.}
\email{ bekker.boris@gmail.com}
\author {Yuri G. Zarhin}
\address{Pennsylvania State University, Department of Mathematics, University Park, PA 16802, USA}
\email{zarhin@math.psu.edu}
\thanks{The second named author (Y.Z.) was partially supported by Simons Foundation Collaboration grant   \# 585711  and the Travel Support for Mathematicians Grant MPS-TSM-00007756 from the Simons Foundation.}

\begin{document}
\begin{abstract}
Let $d\geq 2$ be a positive integer, $K$  an algebraically closed field of characteristic not dividing $d$, and $n\geq d+1$ a positive integer prime to $d$, $f(x)\in K[x]$ a degree $n$ monic polynomial without repeated roots, $C_{f,d}: y^d=f(x)$ the corresponding smooth plane affine curve over $K$, and $\mathcal{C}_{f,d}$ a smooth projective  model of $C_{f,d}$. Let $J(\mathcal{C}_{f,d})$ be the Jacobian of $\mathcal{C}_{f,d} $. We identify $\mathcal{C}_{f,d}$ with the image of its canonical embedding into $J(\mathcal{C}_{f,d})$ (such that the infinite point of $\mathcal{C}_{f,d}$ goes to the zero of the group law on $J(\mathcal{C}_{f,d})$).  Earlier the second named author proved that
 if $d=2$ and $n=2g+1 \ge 5$,
 then  the genus $g$ hyperelliptic  curve $\mathcal{C}_{f,2}$  contains no torsion  points  of orders  lying between $3$ and $n-1=2g$.

In the present paper we generalize this result to the case of arbitrary $d$. Namely,
we prove that if $P$ is a torsion point of order $m>1$ on $\mathcal{C}_{f,d}$, then either $m=d$ or $m\geq n$. We also describe all  curves $\mathcal{C}_{f,d}$ having a torsion point of order $n$.

\end{abstract}


\dedicatory{Dedicated to George Andrews and Bruce Berndt for their 85th birthdays}
\maketitle

\section{Introduction}
\label{introd}
Let $K$ be an algebraically closed field, and $d$ and $n$ are positive integers. In what follows we assume that
$$ 2\leq d\leq n-1, \quad (\fchar(K), d)=1, \quad (n,d)=1.$$
We write $\mu_d$ for the multiplicative (order $d$) cyclic group of $d$th roots of unity in $K$.
Let $$f(x)=\prod_{j=1}^n(x-w_j)\in K[x]$$
be a degree $n$ monic polynomial, where its roots
$w_1,\ldots,w_n$ are distinct elements of $K$, and $C_{f,d}: y^d=f(x)$ is a smooth plane affine curve over $K$.

 The projective closure $\bar C_{f,d}$  has exactly  one point at infinity, which may be singular. We  denote it by $\infty$. The group  $\mu_d$ %of order $d$
 acts on $\bar C_{f,d}$ by sending each point $(x,y)$ of $C_{f,d}$ to $(x,\gamma y)$, where $\gamma \in \mu_d$, and fixing $\infty$.
Let $\CC_{f,d}$ be the normalization of $\bar C_{f,d}$, which is a smooth projective model of $\bar C_{f,d}$. We have a map  $\phi:\CC_{f,d}\to \bar C_{f,d}$  which yields an isomorphism between $C_{f,d}$ and $\phi^{-1}(C_{f,d})$.
Since $(n,d)=1$ it is known \cite{Gal} that $\phi^{-1}(\infty)$ contains only one point $O$, and so the map $\phi$ is bijective. In what follows we identify the points of $\CC_{f,d}$ with the corresponding points of $\bar C_{f,d}$. In particular, slightly abusing notation, we  will write $\infty$ for $O$.
The action of $\mu_d$ on $\bar C_{f,d}$ yields an action of $\mu_d$ on $\CC_{f,d}$  such that
the quotient
$\CC_{f,d}/\mu_d$ is isomorphic to the projective line $\mathbb P^1$. We have a cover $\CC_{f,d}\to \mathbb P^1$ of degree $d$ with $n+1$ ramification points. By the Hurwitz formula,  the genus $g$ of $\CC_{f,d}$ is $(n-1)(d-1)/2$.

Let  $J(\CC_{f,d})$ be the Jacobian of $\CC_{f,d}$, which is a $g$-dimensional abelian variety over $K$. There is a canonical $K$-regular embedding $\alb_O:\CC_{f,d} \to J(\CC_{f,d})$ that sends $O$ to the zero of the group law on $J(\CC_{f,d})$   and each point $P \in \CC_{f,d}(K)$ to the linear equivalence class of the divisor $(P)-(O)$. We identify $\CC_{f,d}$ with its image in $J(\CC_{f,d})$. If $m>1$ is an integer, then by
 an $m$-packet on $\CC_{f,d}$ we mean a nontrivial $\mu_d$-orbit consisting of $m$-torsion $K$-points of $\CC_{f,d}$. (Each such packet consists of $d$ distinct points.)

A celebrated theorem of Raynaud (Manin--Mumford conjecture) implies that if $g>1$ and $\fchar (K)=0$, then the set of torsion points in $\CC_{f,d}(K)$ is finite \cite{Raynaud}. (This assertion does not hold in prime characteristic: e.g., if $K$ is an algebraic closure of a finite field, then $\CC_{f,d}(K)$ is an infinite set that consists of torsion points; it also does not hold if $\fchar (K)=0$ and $g=1$, in which case $\CC_{f,d}(K)$ has a point of any prescribed order.) The importance of  determining explicitly the finite set occurring   in Raynaud's theorem
for specific curves was stressed  by R. Coleman,  K. Ribet, and other authors (for more details, see \cite{Ribet}, \cite{Tsermias}).

In particular, it is  natural  to ask which numbers can be the orders of torsion points in $\CC_{f,d}(K)\subset J(K)$.
More precisely:
\begin{itemize}\item
    does  there exist (for   given $n, d$, $K$, and a positive integer $m>1$) a  curve $\CC_{f,d}$ over $K$ containing a torsion point of order $m?$

    \item if such a curve exists, then how many   points of order $m$ it may contain?
  \end{itemize}

If $\fchar(K)\ne 2$ and $d=2$, then $\CC_{f,d}=\CC_{f,2}$ is  a  {\sl hyperelliptic} curve of genus  $g=(n-1)/2$, and $O$ is  one of the {\sl Weierstrass points} of $\CC_{f,2}$;
it is well known that the torsion points of order $2$ in $\CC_{f,d}(K)$ are  the remaining (different from $O$) $(2g+1)$ Weierstrass points on $\CC_{f,d}(K)$.

  It was proven by J. Boxall and D. Grant in \cite{Box1} that
for $g=2$ there are no points of order $3$ or $4.$

  The second named author proved in   \cite[Theorem 2.8]{Zarhin} that   $\mathcal \CC_{f,2}(K)$  does {\sl not} contain a point of order $m$ if $g \ge 2$ and
 $3\leq m\leq 2g=n-1$.
It is also natural to describe (parametrize)  curves $\CC_{f,d}$ having a point  of a given order on its Jacobian. The first nontrivial case for $d=2$ was considered in \cite{BZ}.
 In the case of $g=2$ such  a study  was done by
 J. Boxall, D. Grant, and F. Lepr\'evost
    in \cite{Box2}, where  a  classification (parameterization) of the genus 2 curves  (up to an isomorphism)  with torsion points of order 5 over algebraically closed fields  was given.  In particular, it was proven in \cite{Box2} that if $\fchar(K)=5$,
     then $\mathcal C(K)$ contains at most 2 points of order 5. These results were generalized by the authors in \cite{BZ} to the case of hyperelliptic curves of arbitrary genus $g \ge 2$
     with torsion points of order $2g+1$.
     In particular, we proved that  if $\fchar(K)=2g+1$,
     then $\mathcal C(K)$ contains at most 2 points of order $2g+1$.

     Notice that the torsion on "generic" superelliptic curves was explicitly described by V. Arul \cite[Sect. 5.3]{Arul}.

Here are  main results of the paper. (If $d=2$, then  the curve $\CC_{f,d}=\CC_{f,2}$ is hyperelliptic and the
  assertions (i) and (ii)(3)  of Main Theorem
 below  become  \cite[Th. 2.8]{Zarhin} and \cite[Th. 5]{BZ} respectively.)

\begin{TheoremA}
\label{TheoremA}
Let $\CC_{f,d}$ be a smooth projective model of the smooth plane affine curve $C_{f,d}: y^d=f(x)$, where $f(x)$ is a monic polynomial of degree $n$ without  repeated roots.
Assume that  $ 2\leq d< n$ and
$$(n,d)=1, \quad (\fchar(K), d)=1.$$
\begin{itemize}
\item[(i)]
  Let $P$ be a point of order $m>1$ on $\CC_{f,d}$. Then either $m=d$ or $m\geq n$.
\item[(ii)]
 %Let %$d=n-1$, i.e.,
%$\CC_{f,d}=\CC_{f,n-1}$ is a smooth projective model of  $C_{f,n-1}: y^{n-1}=f(x)$.
Suppose that  one of the following conditions holds.
\begin{enumerate}
\item[(1)]
$\fchar(K)=0$ and $d$ does not divide $n-1$;
\item[(2)]
$\fchar(K)=0$,  $d=n-1$, and $n\geq5$;
\item[(3)]
$\fchar(K)$ is a prime $p$ and $n$ is a power of $p$.
\end{enumerate}
%If
%$$n\geq5, \quad \fchar(K)=0$$
Then there is at most one $n$-packet on $\CC_{f,d}$, i.e., the number of points of order $n$ on  $\CC_{f,d}$
is either $0$ or $d$.
\item[(iii)]
Let $n=4$ and $d=n-1=3$.  Suppose that $\fchar(K) \ne 2,3$.
Then there are precisely two monic quartic polynomials $f(x) \in K[x]$ without repeated roots
such that all the $K$-points on  the curve  $\CC_{f,3}$ with $x$-coordinates $0$ and $-1$
have order $4$. Namely,
$$f(x)=x^4+2\left(x+\frac{3+\sqrt{3}}{6}\right)^3=(x+1)^4+2 \left(-x+\frac{-3+\sqrt{3}}{6}\right)^3.$$
Two choices of $\sqrt{3} \in K$ give us two polynomials $f(x)$.
%In particular,  $\CC_{f,3}$ has at least two $4$-packets.
%Then $\CC_{f,3}$ is isomorphic to a smooth projective model
%of a curve  $y^3=x^4+(ax+b)^3$ for some $a,b \in K$.
\end{itemize}
\end{TheoremA}

The paper is organized as follows. Section \ref{prel} contains useful elementary results about divisors on $\CC_{f,d}$.   In Section \ref{rest} we discuss possible ``small'' orders of torsion points on $\CC_{f,d}$, actually proving  the first assertion of Main Theorem (Theorem \ref{thm1}). In particular, we classify all points of order $d$.
In Section \ref{orderN}  we discuss torsion points of order $n$, give examples. In Section \ref{picard} we discuss torsion points of order $n=4$ when $d=3$
(i,e., when the curve $\CC_{f,3}$ is a Picard curve of genus 3)  and prove Main Theorem.

{\bf Acknowledgements} We are grateful to the referee, whose thoughtful comments helped to improve the exposition.

\section{Preliminaries}
\label{prel}
Let $K$ be an algebraically closed field. Let $f(x)=\prod_j(x-w_j)$ be a separable degree $n$ polynomial over $K$, and $C_{f,d}: y^d=f(x)$ a smooth plane affine curve over $K$, where $ 2\leq d\leq n-1$ and $(\fchar (K), d)=1$. Let $\CC_{f,d}$ be a smooth projective model of $C_{f,d}$.
Denote by $\zeta$  a primitive  $d$-th root of 1 and by $\omega$ the regular map $C_{f,d}\to C_{f,d}$ that sends a point $P=(a,c)$ to $\omega P=(a,\zeta c)$ and fixes $\infty$.  Clearly $\omega$ induces an action of  $\langle \omega\rangle$ on $\CC_{f,d}$. By linearity
we obtain an action of $\langle \omega\rangle$ on the divisors of $\CC_{f,d}$. We note that for an arbitrary function $F$ on $\CC_{f,d}$ we have
\beq\label{aut}\div  (\omega^{\ast}(F))=\omega(\div   F),\eeq
where $\omega^{\ast}$ is the map corresponding to $\omega$ on the function field of $\CC_{f,d}$.
The map $\omega$ is an automorphism of the divisor group $\Di(\CC_{f,d})$ respecting the linear equivalence.

  \begin{lem}\label{l1}
Let $F(x)=\prod_j(x-a_j)^{c_j}$, where $a_j\in K$ are distinct,  $c_j$ are positive integers,
 be an arbitrary nonconstant polynomial in $K[x]$
of positive degree
$$\deg(F)=\sum_j c_j.$$
 Let $b_j\in K$ be such that  $f(a_j)=b_j^d$.
 Let us consider $F(x)$ as the rational function on $\CC_{f,d}$.
Then
 $$\div(F(x))=\sum\limits_{j}c_j\left(\left(\sum_{i=0}^{d-1}(\omega^i P_j)\right)-d(\infty) \right),$$ where $P_j=(a_j,b_j)$ is a point of $\CC_{f,d}$.
 In particular, $\infty$ is the only pole of  $F(x)$  and its multiplicity is $d \deg(F)$, which is divisible by $d$.
 \end{lem}
 \begin{proof}
If $b_j\neq0$, then the function $x-a_j$ is a local parameter at the points $\omega^i P_j$,  and
\beq
\label{eq1}
\div (x-a_j)= \left(\sum_{i=0}^{d-1}(\omega^iP_j)\right)-d(\infty).
\eeq
If $b_j=0$, then $P_j= \omega^iP_j$ for each $i$, and  $y$ is a local parameter at $P_j$. Since $$x-a_j=y^d\frac{x-a_j}{f(x)}$$ has a zero of multiplicity $d$ at $P_j$, equation \eqref{eq1} is  valid also in this case, which implies the statement of the lemma.

\end{proof}

\begin{lem}\label{l2}
Let $D$ be an arbitrary divisor of degree 0 on $\mathcal C_{f,d}$. Then the divisor $\sum\limits_{i=0}^{d-1}\omega^iD$ is principal.
\end{lem}
\begin{proof}
Let $D=\sum_jc_j(P_j)$
be a divisor of degree 0 on $\mathcal C_{f,d}$.
Since $\sum_jc_j=0$ we have
$$D=\sum\limits_jc_j(P_j)-\sum\limits_jc_j(\infty)=
\sum\limits_jc_j\left((P_j)-(\infty)\right).$$
If some $P_j$ is $\infty$, the corresponding term can be omitted  and we may assume that $P_j\neq \infty$ for all $j$.
Then
$$\begin{aligned}\sum_{i=0}^{d-1}\omega^iD &=\sum_{i=0}^{d-1}\omega^i\sum\limits_jc_j\left((P_j)-(\infty)\right)
 =\sum_{i=0}^{d-1}\sum\limits_jc_j\left(\omega^i(P_j)-\omega^i(\infty)\right)\\
&=\sum\limits_{j}c_j\left(\sum_{i=0}^{d-1}(\omega^iP_j)-n(\infty) \right)=\sum\limits_{j}c_j \div(x-a_j)=
\div F(x),\end{aligned}$$ where
$P_j=(a_j,b_j)$ and $F(x)$ is as in Lemma \ref{l1}.
\end{proof}

\section{Restrictions on the torsion on $\CC_{f,d}$}
\label{rest}

Let $P$ be a point of order $m>1$ on $\mathcal C_{f,d}$. Assume that $m\leq n$. We have   \beq\label{order}m(P)-m(\infty)=m((P)-(\infty))=\div(h), \eeq
where $h$ is a rational function on $\mathcal C_{f,d}$ that is {\sl not} a constant.
The point $\infty$ is a single pole of $h$ and it has multiplicity $m$. Therefore $h$ is a non-constant polynomial in $x,y$. Since $y^d=f(x)$, where $f(x)$ is a polynomial of degree $n$,  $h$ can be represented by a polynomial of the form
\beq \label{1} h=s_1(x)y^{d-1}+s_2(x)y^{d-2}+\cdots+s_{d-1}(x)y-v(x), \,\text {where}\eeq
$$s_1(x),s_2(x),\ldots, s_{d-1}(x),v(x)\in K[x].$$

\begin{Rem}
\label{hEQv}
Let us assume that all $s_1,s_2,\ldots, s_{d-1}$ are zero, i.e.,
$$h=-v(x),$$
where $v(x)$ is a non-constant polynomial in $x$.
%Applying Lemma \ref{l1} to $F(x)=v(x)$,
Since  the only pole of the rational function $x$ on $\mathcal C_{f,d}$ is $\infty$
and its multiplicity is $d$,
%has a pole of multiplicity $d$ at $\infty$
 we get
\beq \label{s_i=0} m=d\deg (v),\eeq
hence $m$ is divisible by $d$.
Since $d\leq m\leq n$, we have
\beq\label{degree v1}
 1\leq\deg (v)\leq \frac{n-1}{d}
\eeq
due to $(n,d)=1$.
\end{Rem}
\begin{Rem}
\label{hNEQv}
Let us assume that at least one of the polynomials $s_1,s_2,\ldots, s_{d-1}$ is nonzero.
We write equation \eqref{1} in the form
\beq \label{2} h=s_{\alpha_1}(x)y^{d-\alpha_1}+s_{\alpha_2}(x)y^{d-\alpha_2}+\cdots+s_{\alpha_k}(x)y^{d-\alpha_k}-v(x), \,\text {where}\eeq
all polynomials $s_{\alpha_i}(x)$ are nonzero and all integers $\alpha_i$ are {\sl distinct positive} and less than $d$.
In particular,
\begin{itemize}
\item[(i)]
Every $\alpha_i$ and $n \alpha_i$ are not divisible by $d$;
\item[(ii)]
all the differences $d-\alpha_i$ are distinct positive integers that are {\sl not} divisible by $d$;
\item[(iii)]
 all
the residues $\alpha_i \ (\bmod \ d)$ are distinct and do not coincide with $0 \ (\bmod \ d)$, which implies that  all
the residues $-n\alpha_i \ (\bmod \ d)$ are distinct and do not coincide with $0 \ (\bmod \ d)$. (Here and in (i) we use that
$(d,n)=1$.)
\end{itemize}
\end{Rem}
\begin{lem}\label{distinctorders}
We keep the assumption of Remark \ref{hNEQv}.
\begin{itemize}
\item[(a)]
If $v(x)$ is not a constant then
all summands on the right-hand side of \eqref{2} have poles  at $\infty$ with distinct positive multiplicities.
\item[(b)]
If $v(x)$ is a constant then all summands on the right-hand side of \eqref{2} except $v(x)$ have poles  at $\infty$ with distinct positive multiplicities.
\end{itemize}
\end{lem}
\begin{proof}
Since the function $x$ has at $\infty$ a pole of multiplicity $d$ and the function $y$ has at $\infty$ a pole of multiplicity $n$, the
multiplicity of the pole of the  function  $s_{\alpha_i}(x)y^{d-\alpha_i}$ at $\infty$ equals to the {\sl positive integer}
\beq
\label{polemult}
d\deg (s_{\alpha_i})+n(d-\alpha_i) \equiv -n \alpha_i \ (\bmod \ d) \not\equiv 0  \ (\bmod \ d).
\eeq
Recall (Remark \ref{hNEQv}(iii)) that all the residues $-n \alpha_i \ (\bmod \ d)$ are distinct and different from $0 \ (\bmod \ n)$. Hence,
all   $d\deg (s_{\alpha_i})+n(d-\alpha_i)$ are {\sl distinct}  integers that are {\sl not} divisible by $d$; in addition,
\begin{equation}
\label{ndA}
d\deg (s_{\alpha_i})+n(d-\alpha_i)\ge n(d-\alpha_i)\ge n\cdot 1=n.
 \end{equation}
On the other hand, either $v(x)$ is a constant or it has  at $\infty$ a pole of positive multiplicity $d \deg(v)$,
which is divisible by $d$. This implies that the multiplicity of the pole $\infty$ of $v(x)$ is divisible by $d$ and therefore does {\sl not} coincide
with any of
$$d\deg (s_{\alpha_i})+n(d-\alpha_i) \not\equiv 0 \ (\bmod \ d),$$
which ends the proof.
\end{proof}

\begin{lem}
\label{rem}
If
$P$ is a point of order $m\le n$ on $\mathcal C_{f,d}$ and
 $m>1$ is not divisible by $d$, then $m=n$  and $m$ coincides with the multiplicity of $\infty$ as a pole of a function $y-v(x)$,
where
$$\begin{aligned} v(x) &\in K[x],  \ \deg (v)\leq \frac{n-1}{d}, \\ \div(y-v(x))= m(P)-m(\infty)=& m((P)-(\infty)).\end{aligned}$$
%$\deg (v)\leq n/d$.
\end{lem}
\begin{proof}
Since $m$ is {\sl not} divisible by $d$,
it follows from Remark \ref{hEQv}
 that we are in the situation of Remark \ref{hNEQv}. Now it follows from  Lemma~\ref{distinctorders} that
\beq\label{3}
m=\max\{d\deg (v), \ \   d\deg (s_{\alpha_i})+n(d-\alpha_i) \ \  \text{ where } 1 \le i \le k \}.
\eeq
Combining \eqref{ndA} and \eqref{3},  we get
$$m \ge d\deg (s_{\alpha_i})+ n(d-\alpha_i) \ge n \quad \forall i.$$
Since $m \le n$, we get $m=n$ and
$$d-\alpha_i=1, \ \   d\deg (s_{\alpha_i})=0 \quad  \forall i.$$
This  means that
$$\alpha_i=d-1, \ \  \text{i.e., } \\ k=1, \ \alpha_1=d-1,$$
$s:=s_1(x)$ is a nonzero constant  and
$$h(x)=s_1(x)y-v(x)=s y-v(x), \quad \div(h)=m(P)-m(\infty)=n(P)-n(\infty).$$
In light of  \eqref{3},
$n=m\ge \deg(v)$, i.e., $\deg(v)\le n/d$.  Since $n$ is {\sl not} divisible by $d$, we get
$$\deg(v) \le \frac{n-1}{d}.$$
Since we are interested in $h$ only up to a nonzero constant, we may
replace $h$ by $s^{-1}h$ and $v(x)$ by $\frac{1}{s}v(x)$ and get
$$h=y-v(x), \quad \div(h)=n(P)-n(\infty),  \quad \deg(v) \le \frac{n-1}{d}.$$
This end the proof.
%$m \ge n$, i.e., $m=n$.
%In particular,
%$$m\geq d\deg (s_{\alpha_i})+n(d-\alpha_i) \;\text{and}\;  m\geq d\deg (v).$$
%Since $\alpha_i<d$, the first inequality implies that $m\geq n$. In light of the assumption $m\leq n$, we obtain $m=n$.

%It follows from the inequality $m\geq d\deg (v)$ that
 %$$\deg (v)\leq \frac{m}{d}=\frac{n}{d}.$$
% Since $\deg (v)$ is an integer and $d$ does {\sl not} divide $n$, we actually have
%$$\deg (v) \le \frac{n-1}{d}.$$

%If $d-\alpha_i>1$ for some $i$ in \eqref{2}, then %at least  one of the numbers
%$$d\deg (s_{\alpha_i})+n(d-\alpha_i)\ge n(d-\alpha_i)>n$$
%is greater than $n$,
%which contradicts the equality $m=n$, in light of \eqref{3}.
%Consequently, all $d-\alpha_i=1$, which implies that
%$k=1$ and  \beq \label{h} h=s(x)y+v(x),\eeq where  $s(x)\in K[x]$ and $$m=\max\{d\deg (s)+n, d\deg (v)\}.$$
%It follows from  $m=n$  that $\deg (s)=0$, i.e.,  $s(x)=s\in K$ is a {\sl nonzero} constant and
%$$h=sy+v(x) \quad \text{ where } \ \deg (v)< \frac{n-1}{d}.$$
%Since $\deg (v)$ is a nonnegative integer and $d$ does {\sl not} divide $n$, we actually have
%$$\deg (v) \le \frac{n-1}{d}.$$
%Since we are interested in $h$ only up to a nonzero constant, we may assume that $h$ has the form $h=y-v(x)$, where $\deg (v)\le (n-1)/d$.
%Therefore, the order of $P$ on $\CC_{f,d}$ coincides with the order of the function $y-v(x)$
%at $\infty$ and
%$m=n$.
\end{proof}
%\begin{Rem}\label{rem}
%We proved that if $P$ is a point of order $m\le n$ and $m$ is not divisible by $d$, then $m=n$  and $m$ coincides with the multiplicity of $\infty$ as a pole of the function $y-v(x)$,
%where
%$$\begin{aligned} v(x) &\in K[x],  \ \deg (v)\leq n/d, \\  m(P)-m(\infty)=& m((P)-(\infty))=\div(y-v(x)).\end{aligned}$$
%$\deg (v)\leq n/d$.
%Since $\deg (v)$ is an integer and $d$ does {\sl not} divide $n$, we actually have
%$$\deg (v) \le \frac{n-1}{d}.$$
%\end{Rem}
\begin{thm}\label{thm1}
Let $\CC_{f,d}$ be a smooth projective model of the smooth plane affine curve $C_{f,d}: y^d=f(x)$, where $f(x)$ is a monic polynomial of degree $n$ without repeated roots.
Assume that  $ 2\leq d< n$ and
$$(n,d)=1, \quad (\fchar(K), d)=1.$$
 Let $P$ be a point of order $m>1$ on $\CC_{f,d}$. Then:
 \begin{itemize}
 \item[(i)]
  Either $m=d$ or $m\geq n$.
  \item[(ii)]
 There are precisely $n$ points of order $d$ on $\CC_{f,d}$, namely,
 $$P_j=(w_j,0) \in C_{f,d}(K)\subset \CC_{f,d}(K) \quad (j=1, \dots, n)$$
 where $w_j \in K$ are the roots of $f(x)$.
 \item[(iii)]
 A point $P=(a,c) \in C_{f,d}(K)\subset \CC_{f,d}(K)$ has order $n$ if and only if
 there exists a polynomial $v(x) \in K[x]$ such that
 $$\deg(v) \le \frac{n-1}{d}, \quad f(x)=(x-a)^n+v^d(x).$$
 %If this is the case, then the point $P=(a,v(a)) \in \CC_{f,d}(K)$ has order $n$.
 \item[(iv)]
 The curve  $\CC_{f,d}$ has a point of order $n$ if and only if there exist $a \in K$
 and a polynomial $v(x) \in K[x]$ such that
 $$\deg(v) \le \frac{n-1}{d}, \quad f(x)=(x-a)^n+v^d(x).$$
 \end{itemize}
\end{thm}
\begin{proof}

(i)  We consider the following two cases separately.

{\bf Case 1}
%\label{case1}
The order $m$ is divisible by $d$. Then $m$ is the order of the pole of $v(x)$ at $\infty$.
Now it follows from Lemma \ref{l1} applied to  $F(x)=v(x)$  that if the support of $\div(v(x))$ contains a point $P$, then
it contains all points $\omega^i P$ in the orbit of $P$ under $\langle\omega\rangle$. But $h=v(x)$,  and by \eqref{order} we have
$$\div(v(x))=m(P)-m(\infty).$$
 It follows that $P=\omega P$,
which means that $P=(w_j,0)$ for some $j$. From Lemma \ref{l2} it follows that $$d\left((P)-(\infty)\right)=\sum_{i=0}^{d-1}\omega^i\left((P)-(\infty)\right)$$
is a principal divisor. Hence $P$  is a point of order dividing $d$. But since $m\geq d$, the order of $P$ is exactly $d$.
We have proven that if $P$ is a point of order $m$ divisible by $d$, then $m=d$ and $P=(w_j,0)$ for some $1\leq j\leq n$.

{\bf Case 2}  Now assume that $m$ is not divisible by $d$. Then it follows from Lemma \ref{rem} that $m=n$ if $m\leq n$. %Consequently, $m\geq n$.
This ends the proof of (i).

(ii) Let $P_j=(w_j,0)$. We have $P_j=\omega^iP_j$, and by Lemma~\ref{l2} the divisor
$$\sum_{i=0}^{d-1}\omega^i\left((P_j)-(\infty)\right)=d\left((P_j)-(\infty)\right)$$ is principal. It follows that the order $m_j$ of $P_j$ in  $J(\CC_{f,d})$ divides $d$.
If $m_j \ne d$, then $m_j<d<n$ and therefore $m_j$ is {\sl not} divisible by $d$. By Case 2, $m_j=n$, which contradicts the inequality $m_j<n$.
This proves that $m_j=d$, i.e. $P_j$ has order $d$. On the other hand, by Case 1 above, every point of order $d$ on $\CC_{f,d}$ coincides with one of $P_j$.
This proves (ii).

(iii) Suppose $P=(a,c)$ is a point of order $n$ on $\CC_{f,d}$.  Then
$$n(P)-n(\infty)=\di (y-v(x)), \;\text{where} \ v(x) \in K[x], \deg (v)\leq(n-1)/d.$$
%Then the divisor  $n(P)-n(\infty)$ is principal and coincides with $\di (y-v(x))$, where $v(x)$ is a polynomial and $\deg (v)\leq(n-1)/d$.
Then the zero divisor of $y-v(x)$ coincides with $n(P)$. In particular, $c=v(a)$.
If $i$  is an integer with $0 \le i \le d-1$, then the point $\omega^i P=(a,\zeta^jic)$ also has order $n$. By \eqref{aut} we have
$$n(\omega^i P)-n(\infty)=\div (\zeta^{-i}y-v(x))=\div (y-\zeta^i v(x)).$$
We have
$$y^d-v^d(x)=(y-v(x))(y-\zeta^1 v(x))\cdots(y-\zeta^{d-1}v(x)).$$
It follows that
$$\begin{aligned}\div(f(x)-v^d(x))=\div(y^d-v^d(x))\\=n\left((P)+(\omega P)+\cdots+(\omega^{d-1}P)\right)-dn(\infty).\end{aligned}$$
Therefore
$f(x)-v^d(x)$ is a polynomial with a single root
$$a=x(P)=x(\omega P)= \dots =x(\omega^{d-1}P)\in K.$$
Since $\deg (v)\leq (n-1)/d$,  the polynomial $f(x)-v^d(x)$ is monic and has degree $n$.
Consequently
$f(x)-v^d(x)=(x-a)^n$ and $f(x)=(x-a)^n+v^d(x)$ as required.

Conversely, let  $\CC_{f,d}$ be  a nonsingular projective model of a smooth plane affine curve $y^d=(x-a)^n+v(x)^d$, where
  $$\deg (v)\leq(n-1)/d, \quad v(a)\neq0.$$
   Let us put $$c:=v(a).$$
   Then  the point $P=(a,c)$ lies in $\CC_{f,d}$. Let us prove that $P$ has order $n$.
  It follows from
 $y^d-v^d(x)=(x-a)^{n}$ that all zeros of  $y-v(x)$  have abscissa $a$. Clearly,
 $P=(a,c)$ is a zero of  $y-v(x)$ but  $\omega^i P=(a,\zeta^i c)$ is {\sl not}  a zero of  $y-v(x)$ for each integer $i$ with
 $d-1 \ge i \ge 1$, because $y-v(x)$ takes the value
 $$\zeta^i c-v(a)=(\zeta^i-1)c+(c-v(a))=
 (\zeta^i-1)c\neq 0$$
 at $\omega^i P$.
 This implies that   $y-v(x)$ has exactly one zero, namely $P$. Obviously, $y-v(x)$ has exactly one pole, namely $\infty$,
 and its multiplicity is $n$. It follows that
 $$\div(y-v(x))=n(P)-n(\infty)=n((P)-(\infty)).$$
 This implies that $P$ has finite order $m$ %in $J(\CC_{f,d})$
  and $m$ divides $n$. Obviously $m\neq 1,d$ and by already proven (i)
 %Theorem~\ref{thm1}
  we get $m=n$.

  (iv) follows readily from (iii).
\end{proof}
 \begin{prop}\label{d points}
 The number of torsion points of order $m\neq d$ on $\CC_{f,d}$ is divisible by $d$.
In particular, if $\CC_{f,d}$ has a point of order $m\neq d$ on $J(\CC_{f,d})$, then it has at least $d$ points of order $m$ on  $J(\CC_{f,d})$.
\end{prop}
\begin{proof}
Let $P$ be a point of order $m$. By $m\neq d$ we have $P=(a,c)$, where $c\neq0$.  Since $\omega$ is an automorphism of $\CC_{f,d}$ that
leaves invariant $\infty$,  the distinct  $d$ points
$$\omega^i P=(a,\zeta^i c) \in \CC_{f,d}(K) \quad (0 \le i \le d-1)$$
also have order $m$.
\end{proof}

\begin{Rem}
\label{change}
Let
$$P=(a_0,c_0), Q=(a_1,c_1)\in     C_{f,d}(K) \subset \CC_{f,d}(K)$$
 be points of order $n$ in $\CC_{f,d}(K)$ with $a_0 \ne a_1$. Let us consider the automorphism of the affine line
 $$S: z \mapsto \lambda z+\mu, \quad \mu=a_0, \ \lambda=a_0-a_1\ne 0.$$
 Then
 $$S(0)=a_0, \quad S(-1)=a_1.$$
 and
 $$\tilde{f}(x):= \lambda^{-n}f(Sx)\in K[x]$$
 is a monic degree $n$ polynomial without repeated roots. Choose
 $$\lambda^{n/d}:=\sqrt[d]{\lambda^n}\in K$$ and consider  the smooth plane affine curve
 $$C_{\tilde{f},d}: w^d=\tilde{f}(z)$$
 and the biregular isomorphism of plane affine curves
 $$\Psi: C_{\tilde{f},d} \to C_{f,d}, \quad (z,w) \mapsto (x,y)=\left(S(z), \lambda^{n/d}w\right),$$
 which extends to the biregular isomorphism $\bar{\Psi}:\CC_{\tilde{f},d} \to \CC_{f,d}$,
 which sends $\infty$ to $\infty$. We have
 $$\begin{aligned}
 \tilde{P}=&\left(0, \ c_0/\lambda^{n/d}\right),  \ \tilde{Q}=\left(-1\ ,c_1/\lambda^{n/d}\right) \in C_{\tilde{f},d}(K)\subset \CC_{\tilde{f},d} (K), \\
&\bar{\Psi}(\tilde{P})=\Psi\left((0,  \ c_0/\lambda^{n/d})\right)=(a_0,c_0)=P,\\
 &\bar{\Psi}(\tilde{Q})=\Psi\left((-1 ,\  c_1/\lambda^{n/d})\right)   =(a_1,c_1)=Q.
 \end{aligned}$$
 This implies that both $\tilde{P}$ and $\tilde{Q}$ are points of order $n$ on
 $\CC_{\tilde{f},d}$ with $z$-coordinates $0$ and $-1$ respectfully.
\end{Rem}

\section{Points of order  $n$}
\label{orderN}

\begin{example}
\label{ex1}
Suppose that $\fchar(K)$ does {\sl not} divide $n$.
Then the polynomial $x^{n}+1$ has {\sl no} repeated roots.  By Theorem~\ref{thm1}(iii) (applied to $a=0$ and $v(x)=\zeta^i$) the  $d$ distinct points $(0,\zeta^i)$ (with $0 \le i \le d-1$) of the curve
$$y^d=x^{n}+1$$
have order $n$ (compare with example~1, p. 6 in \cite{BZ}, where the case  $d=2$ was discussed).

\end{example}
\begin{example}\label{ex2}
Suppose that $\fchar(K)$ divides  $n$.  Choose a {\sl nonzero} $b \in K$.
Then the polynomial $f(x)=x^{n}+(bx+1)^d$ has no repeated roots. Indeed, we have
$f^{\prime}(x)=bd(bx+1)^{d-1}$. So, if $x_0$ is a root of $f^{\prime}(x)$, then $b x_0+1=0$,
which implies that $x_0 \ne 0$ and
$$f(x_0)=x_0^{n}+(bx_0+1)^d=x_0^{n} \neq 0.$$
This proves that $f(x)$ has {\sl no} repeated roots. Applying Theorem \ref{thm1}(iii) to $a=0$ and $v(x)=bx+1$,
we conclude that the curve
$$y^d=x^{n}+(bx+1)^d$$ has  a torsion point $P=(0,1)$ of order $n$  (compare with example~2, p.6 in \cite{BZ}).  By Theorem~\ref{thm1}(iii) the curve has at least $d$ points of order $n$, namely $(0,\zeta^i)$, where $0\leq i\leq d-1$.
\end{example}
\begin{example}\label{ex3}
Suppose that
$$n=4, \ d=3, \quad\fchar(K)\ne 2,3.$$
Choose $\sqrt{3} \in K$. Then
$$\begin{aligned}
f(x)=x^4+2 x^3+\left(\sqrt{3}+3\right)x^2+(\sqrt{3}+2)x+\left(\frac{5}{18}\sqrt{3}+\frac{1}{2}\right)\\
=x^4+2\left(x+\frac{3+\sqrt{3}}{6}\right)^3=(x+1)^4+2 \left(-x+\frac{-3+\sqrt{3}}{6}\right)^3.\end{aligned}$$
So, we have (for any choice of $\sqrt[3]{2}\in K$)
$$f(x)=x^4+v_0^3(x)=(x+1)^4+v_1^3(x)$$ where
$$ v_0(x)=\sqrt[3]{2}\left(x+\frac{3+\sqrt{3}}{6}\right), \ v_1(x)=\sqrt[3]{2}\left(-x+\frac{-3+\sqrt{3}}{6}\right).$$
By Theorem \ref{thm1}, we get six points
$$\left(0, v_0(0)\right)=\left(0, \ \sqrt[3]{2}\ \frac{3+\sqrt{3}}{6}\right), \quad
\left(-1, v_1(-1)\right)=\left(-1, \sqrt[3]{2}\ \frac{3+\sqrt{3}}{6}\right)$$
of order $4$ in $\mathcal{C}_{f,3}(K)$. (Recall that we have three choices of $\sqrt[3]{2}\in K$.)
See also Theorem \ref{4,3} and its proof below.
\end{example}
\begin{thm}
\label{thm3}
Suppose that %$n \ge 3$ and
%$$n \ge 3, \ \ 2 \le d<n,$$
$$2 \le d <n,  \quad (d,n)=1, \quad (d, \fchar(K))=1.$$
Let $\CC_{f,d}$ be a smooth projective model of the smooth plane affine curve $C_{f,d}: y^d=f(x)$, where $f(x)$ is a monic polynomial of degree $n$ without repeated roots.

Assume that one of the following two conditions holds.

\begin{itemize}
\item[(i)]
$p:=\fchar (K)$ is a prime and $n$ is a power of $p$.
\item[(ii)]
$\fchar(K)=0$ and $d$ does not divide $n-1$.
\end{itemize}

Then the number of points of order $n$ on $\CC_{f,d}$ is either $0$ or $d$.
\end{thm}

\begin{proof}
Suppose that  there are at least two $n$-packets on $\CC_{f,n-1}$, i.e., there are
points $(x_0, y_0), (x_1,y_1) \in C_{f,n-1}(K)$ of order $n$ such that $x_0 \ne x_1$.
It follows from Theorem \ref{thm1} (iii) that there are polynomials $v_0(x), v_1(x) \in K[x]$
such that
$$\begin{aligned}f(x)&=(x-x_0)^n+v_0(x)^d=(x-x_1)^n+v_1(x)^d,\\ &\deg(v_0), \ \deg(v_1) \le (n-1)/d.\end{aligned}$$
This implies that
\begin{equation}
\label{x0x1}
(x-x_1)^n-(x-x_0)^n=v_0^d(x)-v_1^d(x)=\prod_{\gamma \in \mu_d}(v_0(x)-\gamma v_1(x)).
\end{equation}

{\bf Case (i)}  $n$ is a power of prime characteristic $p$, the LHS is just $(x_0-x_1)^n$, which is a nonzero constant.
On the other hand, the RHS splits into a product
$$v_0^d(x)-v_1^d(x)=\prod_{\gamma \in \mu_d}(v_0(x)-\gamma v_1(x)).$$
This implies that each factor $v_0(x)-\gamma v_1(x)$ is a constant polynomial, i.e.,
$$v_0(x)-\gamma v_1(x) \in K \quad \forall \gamma \in \mu_d.$$
Since $d \ge 2$,  there is  $\gamma \in \mu_d$ such that $\gamma \ne 1$. This implies that
$$v_0(x)-v_1(x) \in K, \quad v_0(x)-\gamma v_1(x) \in K.$$
It follows that
$(\gamma-1)v_1(x) \in K$, i.e., $v_1(x)$ is a constant, say, $\lambda \in K$.
It follows that
$$f(x)=(x-x_1)^n+\lambda^d.$$
The properties of $n$ imply that $f^{\prime}(x)=0$, and therefore $f(x)$ has a repeated root,
which is not the case. The obtained contradiction proves the desired result.

{\bf Case (ii)} Since $\fchar(K)=0$, the RHS of \eqref{x0x1} is a degree $n-1$ polynomial with leading coefficient
$n(x_0-x_1)$.  Since $d$ does not divide $(n-1)$, the ratio $(n-1)/d$ is not an integer. This implies that
$$\deg(v_0), \ \deg(v_1) < (n-1)/d.$$
It follows that
$$\deg(v_0-\gamma v_i)< (n-1)/d \quad \forall \gamma \in \mu_d.$$
This implies that the RHS of \eqref{x0x1} is a polynomial of degree
$$< d \cdot (n-1)/d=n-1.$$
Since the degree of the LHS is $(n-1)$, which is greater than the degree of the RHS, we get a contradiction, which proves the desired result.
\end{proof}

%\end{proof}

\begin{thm}
\label{thm40}
Let $\CC_{f,n-1}$ be a smooth projective model of the smooth plane affine curve $C_{f,n-1}: y^{n-1}=f(x)$, where $f(x)$ is a monic polynomial of degree $n$ without repeated roots. Assume that  $n\geq5$ and $\fchar (K)=0$. Then the number of points of order $n$ on $\CC_{f,n-1}$ is either $0$ or $n-1$.
\end{thm}

For the proof we need the following auxiliary statement.

\begin{lem}\label{gen}
 Let $n>4$, $n\neq 6$ and $\fchar(K)=0$.
 Let $\eps\in K$ be a primitive $n$th root of unity.
 Let $m$ be a positive integer such that  $(m,n)=1$. Let $\alpha_0,\alpha_1,\alpha_2,\alpha_3\in \BQ(\zeta_m)$.
Then there is no linear fractional transformation $T$ over $K$ such that
$$T(\eps^i)=\alpha_i, \quad 0\leq i\leq3.$$
\end{lem}
We will  also need the following assertion that will be proven at the end of this section.
\begin{lem}
\label{honest}
Let $n \ge 3$ be an integer.
Suppose that $\fchar(K)$  does not divide  $n-1$.
% neither $n$ nor $n-1$. (E.g., $\fchar(K)=0$.)
Let $\bea,\beb,\bec,\bed$ be elements of $K$ such that
$$x^n+(\bea x+\beb)^{n-1}=(x+1)^n +(\bec x+\bed)^{n-1}.$$
Let us assume that the polynomial
$$f_{\bea,\beb;\bec,\bed}(x):=x^n+(\bea x+\beb)^{n-1}=(x+1)^n +(\bec x+\bed)^{n-1}$$
has no repeated roots.
Then:
\begin{itemize}
\item[(i)]
$\bea^{n-1}=n+\bec^{n-1}, \quad \beb^{n-1}=\bed^{n-1}; \quad \beb \ne 0, \ \bed \ne 0.$
\item[(ii)]
Assume additionally that  $\fchar(K)$  does not divide  $n$. Then
the polynomials $\bea x+\beb$ and $\bec x+\bed$ are not proportional, i.e.,
$$\bea\bed-\beb\bec \ne 0$$
and  the map
 $$T: \mathbb{P}^1(K) \to  \mathbb{P}^1(K), \quad z \mapsto
 \frac{\bea z+\beb}{\bec z+\bed}$$
 is an honest linear fractional transformation of the projective line over $K$.
 In addition,
 $$T(\Eta_n)=\mu_{n-1},$$
 %\textrm{E}_n$$
 where $\Eta_n\subset K$ is the $(n-1)$-element set of roots of  the polynomial $(x+1)^n-x^n$
 while  $\mu_{n-1}\subset K$ is the $(n-1)$-element   set of roots of  the polynomial $x^{n-1}-1$, which is actually the multiplicative group of $(n-1)$th  roots of unity in $K$.
\end{itemize}
\end{lem}
\begin{proof}[Proof of Lemma \ref{gen}]
 Assume the contrary. Let $T$ be a linear fractional transformation over $K$ sending $\eps^i$ to $\alpha_i$ for each $i=0,\ldots,3$.  Let us compute the cross-ratio $r\ne 1$ of the quadruple $\{1,\eps,\eps^2,\eps^3\}$.  We get
 $$r=\frac{(\eps^2-1)(\eps^3-\eps)}{(\eps^2-\eps)(\eps^3-1)},$$
 i.e.,
 $$r (\eps^2-\eps)(\eps^3-1)-(\eps^2-1)(\eps^3-\eps)=0.$$
 This implies that
 $$v(t):=r(t^2-t)(t^3-1)-(t^2-1)(t^3-t) \in K[t]$$
 is a degree $5$ polynomial with leading coefficient $r-1$ and $v(\eps)=0$.
 Notice that
 $$v(t)=t(t-1)^2\cdot w(t),$$
 where
 $$w(t)=r(t^2+t+1)-(t+1)^2$$
 is a {\sl quadratic} polynomial  with leading coefficient $r-1$ and $w(\eps)=0$.
Recall that $\mathrm{char}(K)=0$. This implies that if $r \in \mathbb{Q}$, then the polynomial
 $w(t)$ lies  in $\mathbb{Q}[t]$  and is divisible by the $n$-th cyclotomic polynomial in $\mathbb{Q}[t]$.
 Since $n>4$ and $n\neq6$, we have  $\vf(n)>2$ and get a contradiction. Therefore, $r\not\in\mathbb Q$.
 On the other hand, $r$ coincides with the cross-ratio of $\{\alpha_0,\alpha_1, \alpha_2,\alpha_3\}\subset \mathbb Q(\zeta_m)$
 and therefore lies in the intersection
 $$\mathbb{Q}(\zeta_{n}) \cap \mathbb{Q}(\zeta_{m})=\mathbb{Q} $$ since $(m,n)=1$.
 Hence, $r$ lies in $\mathbb{Q}$, which gives us a contradiction.
\end{proof}

\begin{proof}[Proof of Theorem \ref{thm40}]
Assume that a point $(x_0, c)\in \CC_{f,n-1}(K)$ has order $n$. Then the points $P_i=(x_0,\zeta^i c)\in \mathcal C_{f,n-1}(K)$, where
 $\zeta$ is a primitive $(n-1)$-th root of 1 and $0\leq i\leq n-2$, also have order $n$.
Let us  prove that there are no other points of order $n$ on $\CC_{f,d}$.
By Theorem~\ref{thm1}(iii) there exists  a polynomial  $v(x)\in K[x]$  such that
   $\deg(v)\leq1$, $v(x_0)\neq0$, and $ f(x)=(x-x_0)^{n}+v^{n-1}(x).$
Since $\deg (v) \leq1$, we have $v(x)=\bea x+\beb$ for some $\bea,\beb\in K$ and $\bea x_0+\beb\neq0$.
Let $(x_1,c_1)\in \CC_{f,n-1}(K)$ be a point of order $n$ distinct from all $P_i$, i.e.,  $x_1\neq x_0$. Then by Theorem~\ref{thm1}(iii) there exists a polynomial  $v_1(x)\in K[x]$  such that
   $$\deg(v_1)\leq1, \ v(a_1)\neq0, \ f(x)=(x-x_1)^{n}+v_1^d(x).$$ Then $v_1(x)=\bec x+\bed$ for some $\bec,\bed\in K$ and $\bec x_1+\bed\neq0$.
   In light of Remark \ref{change}, we
%We obviously
may assume that
$$x_0=0, \quad x_1=-1.$$
Then  $f(x)=x^n+(\bea x+\beb)^{n-1}$ and $f(x)=(x+1)^n+(\bec x+\bed)^{n-1}$ for some $\bea,\beb,\bec,\bed\in K$.
Let us consider the  polynomial
\begin{equation}
\label{h(x)}
h_n(x):=(x+1)^n-x^n=(\bea x+\beb)^{n-1}-(\bec x+\bed)^{n-1}\in K[x].
\end{equation}
Recall that $\fchar(K)=0$. Then one may easily check that the polynomial $h_n(x)=(x+1)^n-x^n$ is a degree $(n-1)$ polynomial
with leading coefficient $n$
that has $n-1$ distinct roots
$$\eta(\gamma)=\frac{1}{\gamma-1}, \quad  \text{ where } \gamma\in \mu_n\subset K, \ \gamma \ne 1$$
(compare with \cite[Sect. 6.1]{BZ}).

 In light of Lemma \ref{honest}, $\bea x+\beb$ is not proportional to $\bec x+\bed$; in particular, either $\bea \ne 0$ or $\bec \ne 0$. In addition, the map
 $$T: \mathbb{P}^1(K) \to  \mathbb{P}^1(K), \quad z \mapsto
 \frac{\bea z+\beb}{\bec z+\bed}$$
 is an honest linear fractional transformation of the projective line.

From \eqref{h(x)}
 we have
\beq\label{epseta} \prod\limits_{i=1}^{n-1}{(x-\eta_i)}=\prod\limits_{i=0}^{n-2}((\bea x+\beb)-\eps^i(\bec x+\bed)),\eeq
where $\eps$ is a primitive $(n-1)$-th root of 1 and $\eta_i$ are roots of $(x+1)^n-x^n$.
It follows from \eqref{epseta} that the inverse transformation $T^{-1}$ must send $\eps^i$ to $\eta_i$, which contradicts
Lemma \ref{gen} if $n-1>4$  and $n-1\neq 6$, i.e. $n>5, n\neq7$, since the roots of $(x+1)^n-x^n$ lie in $\mathbb Q(\zeta_n)$ and $m:=n-1$ is relatively prime to $n$.
%$(n,n-1)=1$.

 It remains to consider the cases $n=5$, $n=7$.
Let us return to equation \eqref{h(x)}.
Substituting $$x=\frac1{t-1} $$ in \eqref{epseta},
we obtain
$$
\left(\frac1{t-1}+1\right)^n-\left(\frac1{t-1}\right)^n= \left(\frac \bea{t-1}+\beb\right)^{n-1}-\left(\frac \bec{t-1}+\bed\right)^{n-1},$$

$$\frac{t^n-1}{(t-1)^n}=\frac{(\beb t+(\bea-\beb))^{n-1}-(\bed t+(\bec-\bed))^{n-1}}{(t-1)^{n-1}},$$

\beq\label{x to t} t^{n-1}+\cdots+t+1=(\beb t+(\bea-\beb))^{n-1}-(\bed t+(\bec-\bed))^{n-1}
\eeq

\beq\label{x to t 2}\prod\limits_{i=1}^{n-1}(t-\eps^i)=\prod\limits_{i=1}^{n-1}(\beb t+(\bea-\beb)-\zeta_i(\bed t+(\bec-\bed)),\eeq
where $\eps$ is a primitive $n$-th root of unity and $\zeta_i$ are all $(n-1)$-th roots of unity.
As above, it is easy  to show that $$S(z)=\frac{\beb z+(\bea-\beb)}{\bed z+(\bec-\bed)}$$  is an honest linear fractional transformation of the projective line.
It follows from equation \eqref{x to t 2} that $S$ must take $\eps^i$ to $\zeta_i$, which contradicts Lemma \ref{gen} if $n>4$, $n\neq 6$. The group $\mu_3$ acts on $\mu_3$ by multiplications.
\end{proof}

\section{Points of order 4 on Picard curves}
\label{picard}
Throughout this section $n=4$ and $d=3$, i.e., $f(x)$ is a quartic polynomial. If $\fchar(K) \ne 2,3$,
then $\CC_{f,3}$ is a so-called {\sl Picard curve} and has genus $3$.

Recall that every curve $ \CC_{f,3}(K)$ is endowed with the action of the group $\mu_3$, which leaves $\infty$ invariant and sends a point $P=(a,c) \in C_{f,3}(K)$ to
$$\gamma(P)=(a,\gamma c) \quad \forall \gamma \in \mu_3.$$

\begin{thm}\label{4,3}
 Assume that   $\fchar (K)\neq 2,3$.  Let  $\Eta=\Eta_3\subset K$ be the $3$-element set of roots of  $(x+1)^4-x^4$
 and  $\mu_3\subset K$ be the $3$-element set of roots of   $x^3-1$.

 Then
 there is a natural bijection $\mathfrak{B}$ between

 a) the set  $\mathcal{B}$ of bijections $\phi:\Eta \to  \mu_3$ between
$\Eta$  and  $\mu_3$.

 b) The set $\mathcal{F}$ of  quartic monic polynomials $f(x)\in K[x]$ without repeated roots
 and $\mu_3$-orbits of pairs of points $P,Q \in \CC_{f,3}(K)$ that enjoy the following properties.
 \begin{itemize}
 \item[(i)]
 Both $P$ and $Q$ have order $4$;
 \item[(ii)]
 $x(P)=0; \ x(Q)=-1$.
 \item[(iii)]
 Both sets $\mathcal{B}$ and $\mathcal{F}$ are endowed with the natural free actions of $\mu_3$ and the bijection $\mathfrak{B}: \mathcal{B} \to \mathcal{F}$ is $\mu_3$-equivariant.
 Namely, each $\delta \in \mu_3$ sends $\phi\in \mathcal{B}$ to the bijection
 $\delta \phi: \eta \mapsto \gamma \cdot \phi(\eta)$.  An element of $\mathcal{F}$  that consists of a polynomial $f(x)$ and an $\mu_3$-orbit of pair $(P,Q)$
 is sent by $\delta$ to the orbit of $(\delta(P),Q)$ with the same $f(x)$.
 \end{itemize}
% such that all the points of the curve
% $C_{f,3}: y^{3}=f(x)$  with $x$-coordinates $0$ and $-1$ have order $4$.

 %In particular, the  curve $\CC_{f,3}$ has at least 6 points of order  $4$.

 The quartic polynomial $f(x)=f_{\phi}(x)$ attached to  $\phi$ is constructed as follows. Choose
  a linear fractional transformation
$$
T_{\phi}(z)=\frac{\bea z+\beb}{\bec z+\bed}
$$
over $K$ such that
$$T_{\phi}(\eta)=\phi(\eta) \quad \forall \eta \in \Eta; \quad \bea^3-\bec^3=4.$$
Then
$$f_{\phi}(x)=x^4+(\bea x+\beb)^3=(x+1)^4+(\bec x+\bed)^3.$$
The points $P$ and $Q$ are defined as follows.
$$P=(0, \bea \cdot 0+\beb)=(0,\beb); \quad Q=(-1, \bec (-1)+\bed)=(-1,\bed-\bec).$$

\end{thm}

\begin{proof} Let $\Eta=\{\eta_1, \eta_2, \eta_3\}$ and $\mu_3=\{\eps_1, \eps_2,\eps_3\}$ be the sets of roots of  $(x+1)^4-x^4$ and $x^3-1$  respectively and $\phi: \Eta\to \mu_3$ a bijection. We choose a linear fractional transformation
$$
T(z)=\frac{\bea z+\beb}{\bec z+\bed}
$$
such that $T(\eta_i)=\phi(\eta_i)$.  Such a transformation does exist and is unique.

If we check that we can choose $\bea,\beb,\bec,\bed$ in $K$ so that $\bea^3-\bec^3=4$ and  prove that for
these $\bea,\beb,\bec,\bed$ the polynomial $f_{\phi}(x)=x^4+(\bea x+\beb)^3$ has no repeated roots, then we get the identity
\beq\label{2 packets 2}
(x+1)^4 - x^4=(\bea x+\beb)^3-(\bec x+\bed)^3,
\eeq
and by Theorem \ref{thm1}(iii) we obtain that the points of the curve $$C_{f_{\phi},3}: y^{3}=f_{\phi}(x)$$ with $x$-coordinates $0$ and $-1$  have order 4.
The discriminant of the polynomial
$$f_{\phi}(x)=x^4+(\bea x+\beb)^3 \in K[x], \quad \mathrm{char}(K) \ne 2,3$$ is
$$-27 \bea^4 \beb^8 + 256 \beb^9=\beb^8 (256\beb-27 \bea^4).$$
In other words, $f_{\phi}(x)$ has a multiple root if  and only if  $\beb =0$ or $\beb=\frac{27}{256}\bea^4$.
Suppose that
$$f_{\phi}(x)=(x+1)^4+(\bec x+\bed)^3=(x+1)^4+(\bec(x+1)+(\bed-\bec))^3.$$
Then either
$$\bed-\bec=0, \ \text{i.e.}, \ \bed=\bec$$
or
$$\bed-\bec=\frac{27}{256}\bec^4, \ \text{i.e.}, \ \bed=\bec+\frac{27}{256}\bec^4.$$
We have
\begin{equation}
\label{main}
4x^3+6x^2+4x+1=(x+1)^4-x^4=(\bea x+\beb)^3-(\bec x+\bed)^3.
\end{equation}

Suppose that $\beb=0$. By \eqref{main}  we have
$$4x^3+6x^2+4x+1=(x+1)^4-x^4=(\bea x)^3-(\bec x+\bed)^3.$$
This means that
$$\bea^3-\bec^3=4, \ 3\bec^2 \bed=-6, \  3\bec \bed^2=-4, \ \bed^3=-1.$$
Multiplying (if necessary) $\bec x+\bed$ by a suitable cube root of unity, we may and will assume that
$\bed=-1$. This implies that
$$3 \bec^2=6, \ 3\bec =-4,$$
which could not be the case. Hence, $ \beb \ne 0$. By the same token,
$\bed-\bec \ne 0$, i.e., $\bed \ne \bec$.

So,

$$\beb=\frac{27}{256}\bea^4, \ \bed=\bec+\frac{27}{256}\bec^4.$$
 By \eqref{main}, we have
$$\begin{aligned}4x^3+6x^2+4x+1=(x+1)^4-x^4\\=\left(\bea x+\frac{27}{256}\bea^4\right)^3-\left(\bec x+\bec+\frac{27}{256}\bec^4\right)^3\\
=\bea^3\left(x+\frac{27}{256}\bea^3\right)^3-\bec^3\left(x+1+\frac{27}{256}\bec^3\right)^3\\=
A\left(x+\frac{27}{256}A\right)^3-C\left(x+1+\frac{27}{256}C\right)^3,\end{aligned}$$
where $A=\bea^3$ and $C=\bec^3$.
It follows that
$$A-C=4,$$
\beq \label{eq4}\frac{81}{256} A^2-3C\left(1+\frac{27}{256}C\right)=6,\eeq

\beq\label{eq5}3 \cdot \frac{27^2}{256^2}A^3-3C\left(1+\frac{27}{256}C\right)^2=4.\eeq
 Plugging  $A=C+4$ in \eqref{eq4}, we get
\beq\label{eq101}\frac{81}{256} (C+4)^2-3C\left(1+\frac{27}{256}C\right)=6,\eeq
\beq \label{eq103}\frac{81}{256}C^2+\frac{81}{32}C+\frac{81}{16}-3C-  \frac{81}{256}C^2=6.\eeq
Plugging  $A=C+4$ in \eqref{eq5}, we get
\beq\label{eq102}3 \cdot \frac{27^2}{256^2}(C+4)^3-3C\left(1+\frac{27}{256}C\right)^2=4,\eeq
Equation \eqref{eq103} has obviously a unique solution $C=-2$ that does not satisfy \eqref{eq102}.

So for each bijection $\phi:\Eta\to \textrm{E}$ we have constructed a quartic monic polynomial $f_{\phi}(x)$ with required properties.
We also constructed the  pair of polynomial $v_0(x)=\bea x+\beb, \ v_1(x)=\bec x+\bed$, which is determined up to multiplication by an element of $\mu_3$.
They give rise to the points of order $4$
$$P=(0, v_0(0))=(0,\beb), \ Q=(-1,\bed-\bec) \in \mathcal{C}_{f,3}(K)$$
with abscissas $0$ and $-1$ respectively. Replacing $\left((v_0(x),v_1(x)\right)$ by $\left(\gamma v_0(x),\gamma v_1(x)\right)$ (where $\gamma \in \mu_3$)
gives us the pair $\gamma(P), \gamma(Q)$  of points of order 4 with abscissas $0$ and $-1$ respectively.
In other words, $\phi$ gives rise to the $\mu_3$-orbit of pairs $(P,Q)$ of points of order $4$ with abscissas $0$ and $-1$ respectively.
Now let us define $\mathfrak{B}(\phi)$ as the element of $\mathcal{F}$ that consists of $f_{\phi}(x)$ and the $\mu_3$-orbit of $(P,Q)$.

For each $\delta \in \mu_3$ the bijection
$$\delta(\phi):  \eta \mapsto \delta\cdot  \phi(\eta)$$
coincides with the restriction to $\Eta$ of the linear fractional map
$$\delta T: T(z)=\frac{\delta(\bea z+\beb)}{(\bec z+\bed)}=\frac{\delta\bea z+\delta\beb}{\bec z+\bed}.$$
This map gives us (actually the same) the  quartic
polynomial
$$f_{\delta (\phi)}x^4+(\delta(\bea z+\beb))^3=x^4+(\bea z+\beb)^3=f_{\phi}(x).$$
On the other hand, the corresponding pair of order $4$ points is $(\delta(P),Q)$.
This proves that the map $\mathfrak{B}$ is $\mu_3$-equivariant.

Conversely, suppose we are given a   monic quartic polynomial $f(x)\in K[x]$ without repeated roots
and points
$$P=(0, c_0), \ Q=(-1, c_1) \in \mathcal{C}_{f,3}(K)$$ of order 4.
% such that these the points
 %of the curve
 %$y^3=f(x)$ with $x$-coordinates $0$ and $-1$
% have order $4$.
 Then by Theorem \ref{thm1}(iii)
 $$f(x)=x^4+(\bea x+\beb)^3=(x+1)^4+(\bec x+\bed)^3, \quad  c_0=\beb, \ c_1=\bed-\bec$$
 for some $\bea,\beb,\bec,\bed\in K$. It follows that the coefficient of $f(x)$ at $x^3$ is
 $\bea^3=4 +\bec^3$. This implies that
 $$\bea^3-\bec^3=4.$$
 By Lemma \ref{honest}, $\bea\bed-\beb\bec \ne 0$ and if we consider the
 linear fractional transformation
 $$T: \mathbb{P}^1(K) \to \mathbb{P}^1(K) , \quad z \mapsto \frac{\bea z+\beb}{\bec z+\bed},$$
 then $T(\Eta)= \mu_3$. In particular, $T$ induces a bijection
 $$\phi_{f,P,Q}:\Eta\to \mu_3, \quad \eta \mapsto T(\eta).$$
 If we replace the pair $(P,Q)$ by any $(\gamma(P),\gamma(Q))$  (where $\gamma\in \mu_3$),
 then the corresponding linear fractional transformation will be (actually the same)
 $$z \mapsto \frac{\gamma(\bea z+\beb)}{\gamma(\bec z+\bed)}=T(z).$$
 This implies that the bijection $\phi_{f,P,Q}$ actually depends only on $f(x)$ and the $\mu_3$-orbit of $(P,Q)$.
 This proves that $\mathfrak{B}$ is a {\sl bijection}.

 %$T$ sending the roots $\eta_1, \eta_2, \eta_3$ of $(x+1)^4-x^4$ to the cubic roots of 1, which yields a bijection $\phi:\Eta\to \textrm{E}$.
 It follows from our considerations that  each $\mu_3$-orbit in $\mathcal{B}$ gives the same quartic polynomial while distinct
 orbits give rise to distinct quartic polynomials. Since the cardinality of  $\mathcal{B}$  is $6$, we obtain that there are precisely two quartic polynomials
 $f(x)$ without repeated roots such that all the points on $\mathcal{C}_{f,3}$ with abscissas $0$ and $-1$ have order $4$.
 Now one should find explicitly  linear fractional transformations that send $\Eta$ to $\mu_3$, which can be done
 by direct calculations. As a result, we get explicit formulas for such $f(x)$.
\end{proof}

\begin{proof}[Proof of Lemma \ref{honest}]
 Let us put
$$f(x):=f_{\bea,\beb;\bec,\bed}(x)=x^n+(\bea x+\beb)^{n-1}=(x+1)^n +(\bec x+\bed)^{n-1}.$$
Then the constant term of $f(x)$ is
$\beb^{n-1}=\bed^{n-1}$. On the other hand, the coefficient of $f(x)$ at $x^{n-1}$ is
$a^{n-1}=n+\bec^{n-1}$. This proves (i).
It also follows that either $\beb=\bed=0$ or $\beb\ne 0, d \ne 0$.
If $\beb=0$, then
$f(x)=x^n+(\bea x)^{n-1}$ has a root $0$ of multiplicity $\ge n-1>1$, which is not true.
Hence,
$$\beb\ne 0, \ \bed \ne 0.$$
This ends the proof of (i).

Let us start the proof of (ii). Suppose that $\bea \bed-\beb \bec=0$, i.e., $\bea \bed=\beb \bec$.
We need to arrive to a contradiction.) Clearly,
$\bea^{n-1} \bed^{n-1}=\beb^{n-1} \bec^{n-1}$. It follows that
$$(n+\bec^{n-1})\beb^{n-1}=\beb^{n-1} \bec^{n-1},$$
i.e.,   $n \beb^{n-1}=0 \in K$. Since neither $n$ nor $\beb$ are zeros in $K$, we get the desired contradiction.  So,
$$\bea \bed-\beb \bec \ne 0$$
and
$$T: \mathbb{P}^1(K) \to  \mathbb{P}^1(K), \quad z \mapsto \frac{\bea z+\beb}{\bec z+\bed}$$
is a bijective self-map.

Let us consider  the polynomial
$$h_n(x):=(x+1)^n-x^n=(\bea x+\beb)^{n-1}-(\bec x+\bed)^{n-1}.$$
Since $\fchar(K)$ does not divide $n$, one may easily check that
$h_n(x)$
is a degree $(n-1)$ polynomial
with leading coefficient $n$
that has $n-1$ distinct roots
$$\eta(\gamma)=\frac{1}{\gamma-1}, \quad  \text{ where } \gamma\in \mu_n\subset K, \ \gamma \ne 1$$
(compare with \cite[Sect. 6.1]{BZ}); notice that none of these roots is $0$.
 Since $n \ge 3$, $h_n(x)$ has more than one roots.

Recall that $\fchar(K)$ does not divide $n-1$, so the multiplicative group $\textrm{E}_n=\mu_{n-1}\subset K$
of $(n-1)$th roots of unity consists of $(n-1)$ elements. We have
$$\begin{aligned}\prod_{\eta\in \Eta_n}(x-\eta)&=(x+1)^n-x^n=(\bea x+\beb)^{n-1}-(\bec x+\bed)^{n-1}\\&=\prod_{\gamma \in \textrm{E}_n}((\bea x+\beb)-\gamma (\bec x+\bed)).\end{aligned}$$
This implies that for each $\eta\in \Eta_n$ there exists $\gamma \in  \textrm{E}_n$ such that
$$(\bea \eta+\beb)-\gamma(\bec \eta+\bed)=0,$$ i.e.,
$$(\bea \eta+\beb)=\gamma(\bec\eta+\bed).$$
We claim that $\bec \eta+\bed\ne 0$. Indeed, if $\bec \eta+\bed=0$, then $\bea \eta+\beb=\gamma \cdot 0=0$. So,
$$ \bea \eta+\beb=\bec \eta+\bed=0,$$
which contradicts  the inequalities
$$\bea \bed-\beb \bec \ne 0, \ \eta \ne 0.$$
So, $\bec \eta+\bed\ne 0 \ne 0$ and therefore
$$\gamma=\frac{\bea \eta+\beb}{\bec \eta+d}=T(\gamma).$$
This implies that $T(\Eta_n)\subset \textrm{E}_n$. Since both finite sets $\Eta_n$ and $\textrm{E}_n$ have the same number of elements
(namely, $n-1$) and $T$ is injective, we get
$$T(\Eta_n)= \textrm{E}_n,$$
which ends the proof of  (ii).
\end{proof}

\begin{proof}[Proof of Main Theorem]
The desired result is an immediate corollary of Theorems \ref{thm1},  \ref{thm3}, \ref{thm40}, \ref{4,3}.

\end{proof}

\vskip .2cm
{\bf Conflict of Interest statement }  Not Applicable

\end{document}